\documentclass[graybox]{svmult}

\usepackage{type1cm}        
\usepackage{makeidx}         
\usepackage{graphicx}        
\usepackage{multicol}        
\usepackage[bottom]{footmisc}

\usepackage{newtxtext}       %
\usepackage{newtxmath}       

\usepackage{url}

\def\RR{\mathbb{R}}

\def\CC{\mathbb{C}}
\def\NN{\mathbb{N}}

\newcommand{\al}{{\alpha}}
\newcommand{\la}{{\lambda}}
\newcommand{\f}{{\varphi}}

\newcommand{\C}{{\mathbb  C}}

\newcommand{\bA}{{{\bf A}}}

\newcommand{\fdot}{\,\cdot\,}

\newcommand{\cH}{\mathcal{H}}

\newcommand{\cC}{\mathcal{C}}
\newcommand{\cB}{\mathcal{B}}
\newcommand{\cF}{\mathcal{F}}
\newcommand{\cA}{\mathcal{A}}

\newcommand{\cD}{\mathcal{D}}

\newcommand{\ft}{\mathfrak{t}}

\newcommand{\cV}{\mathcal{V}}

\DeclareMathOperator{\dom}{dom}
\DeclareMathOperator{\Ran}{Ran}

\newcommand{\ci}[1]{_{ {}_{\scriptstyle #1}}}
\newcommand{\ti}[1]{_{\scriptstyle \text{\rm #1}}}

\makeindex             

\begin{document}

\title*{Perspectives on General Left-Definite Theory\thanks{The authors would like to congratulate Lance Littlejohn on the occasion of his 70th birthday, and thank him for all of the help over the years. Lance provides the inspiration and encouragement to dream big.}}
\titlerunning{Perspectives on Left-Definite Theory} 
\author{Dale Frymark and Constanze Liaw}
\institute{Dale Frymark \at Nuclear Physics Institute, Department of Theoretical Physics, Czech Academy of Sciences, 25068 Řež, Czech Republic. \newline
\email{frymark@ujf.cas.cz}
\and Constanze Liaw \at University of Delaware, Department of Mathematical Sciences, 501 Ewing Hall, Newark, DE 19716, USA, and \newline
CASPER, Baylor University, One Bear Place \#97328,      
 Waco, TX 76798, USA. \newline
\email{liaw@udel.edu}}
%
%
\maketitle

\abstract*{In 2002, Littlejohn and Wellman developed a celebrated general left-definite theory for semi-bounded self-adjoint operators with many applications to differential operators. The theory starts with a semi-bounded self-adjoint operator and constructs a continuum of related Hilbert spaces and self-adjoint operators that are intimately related with powers of the initial operator. The development spurred a flurry of activity in the field that is still ongoing today. \newline\indent
The main goal of this expository (with the exception of Proposition \ref{p-new}) manuscript is to compare and contrast the complementary theories of general left-definite theory, the Birman--Krein--Vishik (BKV) theory of self-adjoint extensions and singular perturbation theory. In this way, we hope to encourage interest in left-definite theory as well as point out directions of potential growth where the fields are interconnected. We include several related open questions to further these goals.}

\abstract{In 2002, Littlejohn and Wellman developed a celebrated general left-definite theory for semi-bounded self-adjoint operators with many applications to differential operators. The theory starts with a semi-bounded self-adjoint operator and constructs a continuum of related Hilbert spaces and self-adjoint operators that are intimately related with powers of the initial operator. The development spurred a flurry of activity in the field that is still ongoing today. \newline\indent
The main goal of this expository (with the exception of Proposition \ref{p-new}) manuscript is to compare and contrast the complementary theories of general left-definite theory, the Birman--Krein--Vishik (BKV) theory of self-adjoint extensions and singular perturbation theory. In this way, we hope to encourage interest in left-definite theory as well as point out directions of potential growth where the fields are interconnected. We include several related open questions to further these goals.}

\begin{acknowledgement}
Since August 2020, Liaw, C. has been serving as a Program Director in the Division of Mathematical Sciences at the National Science Foundation (NSF), USA, and as a component of this position, she received support from NSF for research, which included work on this paper. Any opinions, findings, and conclusions or recommendations expressed in this material are those of the authors and do not necessarily reflect the views of the NSF.
\end{acknowledgement}

\section{Introduction}\label{s-INTRO}

The development of general left-definite theory by Littlejohn and Wellman had vast repercussions for differential operators. It extended the important concept of a left-definite Hilbert space to a continuum of associated Hilbert spaces\footnote{All Hilbert spaces considered are assumed to be separable.}. We appeal to differential operators to further explain the motivation for this extension from the classical left-definite space by following \cite{LW02} below. Note that some knowledge of self-adjoint extension theory is assumed throughout the manuscript, the reader is encouraged to consult the Appendix for some basic definitions and notions. The references \cite{AG, N} can be used for more details. For some locally integrable positive weight function $w$, let
\begin{align}\label{e-basic}
    \ell[y](x)=\la w(x)y(x),
\end{align}
where $\ell$ is Lagrangian symmetric differential expression of order $2n$ given by
\begin{align}\label{e-lagrangian}
\ell^n[f](x)=\sum_{j=1}^n(-1)^j(a_j(x)f^{(j)}(x))^{(j)}, \text{ } x\in(a,b),
\end{align}
where $-\infty\leq a<b\leq\infty$. For simplicity, we assume each coefficient $a_j(x)$ is smooth and positive on $(a,b)$. Due to the presence of $w(x)$ on the right hand side of equation \eqref{e-basic}, $\cH=L^2[(a,b),w(x)]$ is referred to as the {\em right-definite} spectral setting for $w^{-1}\ell$. Two well-known formulas play a role in defining the left-definite space.

For $f,g\in\cD\ti{max}$, the maximal domain of $\ell$, Green's formula says that
\begin{align}\label{e-greens}
    \int_a^b\ell[f](x)\overline{g}(x)dx=\int_a^bf(x)\overline{\ell[g]}(x)dx+[f,g](x)\big|_a^b,
\end{align}
where $[\fdot,\fdot]$ is the sesquilinear form for $\ell$. Dirichlet's formula is
\begin{equation}\label{e-dirichlet}
\begin{aligned}
\int_a^b\ell[f](x)\overline{g(x)}dx=\sum_{j=0}^n\int_a^b a_j(x)&f^{(j)}(x)\overline{g}^{(j)}(x)dx \\
&+\llbracket f,g\rrbracket(x)\big|_a^b,
\end{aligned}
\end{equation}
where $\llbracket\fdot,\fdot\rrbracket$ is a bilinear form closely related to the sesquilinear form $[\fdot,\fdot]$. Let ${\bf A}$ be a self-adjoint operator acting via $\ell$ on a domain $\cD({\bf A})$ such that for all $f,g\in\cD({\bf A})$ the bilinear form $\llbracket f,g\rrbracket(x)\big|_a^b=0$. In this case, equation \eqref{e-dirichlet} simplifies to 
\begin{align}\label{e-setup}
    \langle {\bf A}f,g\rangle_{\cH}=\int_a^b\ell[f](x)\overline{g(x)}dx=\sum_{j=0}^n\int_a^b a_j(x)&f^{(j)}(x)\overline{g}^{(j)}(x)dx.
\end{align}
Such an operator ${\bf A}$ is then semi-bounded (see eq.~\eqref{e-semibdd}) thanks to the positivity of the coefficient functions $a_j(x)$. The expression $\ell$ thus generates an inner product defined by the right hand side of equation \eqref{e-setup}. Explicitly, for $f,g\in\cD({\bf A})$,
\begin{align}\label{e-first}
\langle {\bf A}f,g\rangle_{\cH}=\langle f,g\rangle_1.
\end{align}
The closure of $\cD({\bf A})$ in the topology generated by the norm $\|\fdot\|_1=\langle\fdot,\fdot\rangle_1^{1/2}$ is then denoted by $H_1$ and referred to as the left-definite setting for $w^{-1}\ell$, with $H_1$ the (first) left-definite space. This is due to the inner product $\langle\fdot,\fdot\rangle_1$ originating from the left hand side of equation \eqref{e-basic}. The terminology itself can be traced back to Weyl in 1910 \cite{W}. 

This first left-definite space presented above is identified by Littlejohn and Wellman as $\cD({\bf A}^{1/2})$, along with a continuum of other associated left-definite spaces in Definition \ref{t-ldinpro}. 
Their new characterization makes connections to other areas of self-adjoint operator theory more clear. The core definitions and results of left-definite theory (we omit the term `general' from now on) are included in Section \ref{s-LD}. A key aspect of the theory is that it includes stability of spectral type, see Theorem \ref{t-leftdefortho}, which can be viewed as an important advantage over the other theories discussed. Since left-definite theory was described in 2002, it has been the focus of many works. Here, we focus on techniques and results recently developed by the authors in order to present several related conjectures and open questions. It is the our hope that researchers may take inspiration from these problems to make contributions to the field.

The Birman--Krein--Vishik (BKV) theory of semi-bounded forms, which puts closed semi-bounded forms into a one-to-one correspondence with self-adjoint operators, agrees with general left-definite theory. In many ways, this theory could also be thought of as a generalization of equation \eqref{e-setup}: identifying $\langle {\bf A}f,g\rangle_{\cH}$ as an inner product in another Hilbert space for general self-adjoint operators ${\bf A}$. BKV theory and its connections to left-definite theory are the subject of Section \ref{s-BKV}. Proposition \ref{p-new} proves that a continuum of semi-bounded closed forms are actually associated with each semi-bounded self-adjoint operator via left-definite theory. Such freedom in choosing a form to work with may provide valuable flexibility in applications. 

There is also a scale of Hilbert spaces that is used in singular self-adjoint perturbation theory, when the perturbation does not lie in the Hilbert space but can be shown to belong to an associated Hilbert space. This scale of Hilbert spaces is generated from a self-adjoint operator, and if this operator is semi-bounded they are equivalent to the continuum of left-definite spaces of the operators. However, the scale of Hilbert spaces is focused on linear bounded functionals which act on these left-definite spaces, thereby corresponding to left-definite spaces with negative indices that are defined through duality. Although independently defined, there are instances in which the definitions of the spaces occurring in left-definite theory are identical to those in the scale of Hilbert spaces. We find this intriguing, especially due to the fact that they are used for different, complementary purposes. The connection between this scale of spaces and left-definite theory is described in Section \ref{s-SCALE}.

One of the main obstacles in both BKV theory and singular perturbation theory is that the associated spaces are difficult to describe with explicit boundary conditions. This is another area where connections to left-definite theory can be of benefit, as there are now clear descriptions even in some hard to handle examples, i.e.~Sturm--Liouville operators with limit-circle endpoints.

Self-adjoint extensions of symmetric Sturm--Liouville operators with limit-circle endpoints were recently characterized in terms of a singular perturbation by the authors in \cite{BFL}. The perturbation heavily exploited the connection between the scale of Hilbert spaces and BKV theory, along with boundary triples and boundary pairs (see e.g.~\cite{BdS}). In particular, we mention that boundary triples are partially defined by the Green's formula in equation \eqref{e-greens} and boundary pairs essentially operate on two key spaces from the scale of Hilbert spaces.

The basics of this perturbation setup are described in Section \ref{s-PERT} in order to formulate a conjecture about how left-definite theory may assist with extending the construction to powers of such operators. 

While the applications in this manuscript are focused on ordinary differential operators, it should be pointed out that left-definite theory does apply to self-adjoint elliptic partial differential operators too. However, to the best knowledge of the authors this fact has not been thoroughly explored.

Overall, the three theories: left-definite, BKV and singular perturbation, are found to be complementary to each other. We hope this manuscript helps readers find new and interesting connections between fields they may or may not have encountered before and sparks their interest in building upon these connections.

\subsection{Notation}

The notation $L\ti{max}=\{\ell,\cD\ti{max}\}$ is used to say that the operator $L\ti{max}$ acts via $\ell$ on the domain $\cD\ti{max}$. Self-adjoint operators are written in bold face, e.g.~${\bf A}$, for emphasis while all other operators are typeset normally.

Due to the three different theories in the manuscript, we clarify the notation for each here to help avoid any confusion. In left-definite theory, the inner product is denoted with a single subscript $r$, i.e.~$\langle\fdot,\fdot\rangle_r$, and the Hilbert space associated with this inner product is written as $\cH_r$ (see Definition \ref{t-ldinpro}). In the BKV theory of semi-bounded forms, the inner product and associated Hilbert space are both denoted with the form as a subscript, i.e.~$\langle\fdot,\fdot\rangle\ci{\ft_{{\bf S}-\gamma}}$ and $\cH\ci{\ft_{{\bf S}-\gamma}}$ (see equation \eqref{e-forminnerprod}). Finally, when discussing the scale of Hilbert spaces associated with a self-adjoint operator ${\bf A}$, we use double subscript $\langle\fdot,\fdot\rangle_{s,-s}$ to denote the duality pairing and, e.g.~$\cH_{-s}({\bf A})$ to denote the desired space (see Definition \ref{d-standardscale}).

\section{Sturm--Liouville Operators}\label{s-BACK}

Many of the theories and conjectures in this manuscript, especially those in Section \ref{s-LD} and \ref{s-PERT}, are illustrated by applications to Sturm--Liouville operators because they are so well studied, see e.g.~\cite{AG, BEZ, E, Z}. We briefly introduce the central concepts of these operators here.

Consider the classical Sturm--Liouville differential equation
\begin{align*}
\dfrac{d}{dx}\left[p(x)\dfrac{df}{dx}(x)\right]+q(x)f(x)=-\lambda w(x)f(x),
\end{align*}
where $p(x),w(x)>0$ a.e.~on $(a,b)$ and $q(x)$ real-valued a.e.~on $(a,b)$, with $a<b$ and $a,b\in\RR\cup\{\pm \infty\}$.
Furthermore, let $1/p(x),q(x),w(x)\in L^1\ti{loc}(a,b)$. 
The differential expression can be viewed as a linear operator, mapping a function $f$ to the function
\begin{align}\label{d-sturmop}
\ell[f](x):=-\dfrac{1}{w(x)}\left(\dfrac{d}{dx}\left[p(x)\dfrac{df}{dx}(x)\right]+q(x)f(x)\right).
\end{align}
This unbounded operator acts on the Hilbert space $L^2[(a,b),w]$, endowed with the inner product 
$
\langle f,g\rangle:=\int_a^b f(x)\overline{g(x)}w(x)dx.
$
In this setting, the eigenvalue problem $\ell[f](x)=\lambda f(x)$ can be considered. However, the operator acting via $\ell[\fdot]$ on $L^2[(a,b),w]$ is not self-adjoint a priori; additional boundary conditions may be required to ensure this property. The definition here is slightly different than that of the differential operator in equation \eqref{e-lagrangian}, where to explain the different left and right-definite settings the weight function was not included. Hence, the operator of interest ended up arising from $w^{-1}\ell$, which agrees with the expression here.

Endpoints are either in the limit-point or limit-circle case depending on whether one or two solutions are in $L^2[(a,b),w]$, respectively. When in the limit-circle case, we assume that endpoints are non-oscillatory, see e.g.~\cite{BEZ, Z} for more.

Furthermore, the operator $\ell^n[\fdot]$ is defined as the order two operator $\ell[\fdot]$ composed with itself $n$ times, creating a differential operator of order $2n$. Every formally symmetric differential expression $\ell^n[\fdot]$ of order $2n$ with coefficients $a_k:(a,b)\to\RR$ and $a_k\in C^k(a,b)$, for $k=0,1,\dots,n$ and $n\in\NN$, has the Lagrangian symmetric form given in equation \eqref{e-lagrangian}.
Further details can be found in \cite{DS, ELT}.

The classical differential expressions of Jacobi, Laguerre and Hermite are all semi-bounded and admit such a representation. Semi-boundedness is defined as the existence of a constant $k\in\RR$ such that for all $x$ in the domain of the operator $A$ the following inequality holds:
\begin{align}\label{e-semibdd}
    \langle Ax,x\rangle\geq k\langle x,x\rangle.
\end{align}
This additional property, combined with self-adjointness, allows for a continuum of nested Hilbert spaces to be defined within $L^2[(a,b),w]$ via non-negative real powers of the expression $\ell$. Indeed, this continuum will provide a Hilbert scale, and many facts about the spectrum and the operators (see e.g.~\cite{DHS, LW13}) can be deduced using this point of view. More details about Hilbert scales can be found in \cite{AK, KP}. This particular Hilbert scale with self-adjoint operators that are semi-bounded is the topic of left-definite theory \cite{LW02}.

\section{Left-definite theory}\label{s-LD}

General left-definite theory uses powers of a semi-bounded self-adjoint differential operator to create a continuum of operators whereupon spectral properties can be studied. This spectral information is invariant in the sense that knowledge about the spectrum of one of the operators in this continuum allows us to obtain insight into all the other operators. We begin by reviewing some of the main definitions and results from the landmark paper by Littlejohn and Wellman \cite{LW02} in order to introduce the theory.

Let $\cV$ be a vector space over $\CC$ with inner product $\langle\fdot,\fdot\rangle$ and norm $\|\fdot\|$. The resulting inner product space is denoted $(\cV,\langle\fdot,\fdot\rangle)$.

\begin{definition}[{\cite[Theorem 3.1]{LW02}}]\label{t-ldinpro}
Suppose ${\bf A}$ on the Hilbert space $\cH=(\cV,\langle\fdot,\fdot\rangle )$ is a self-adjoint operator that is bounded below by $kI$, where $k>0$. Let $0<r\in\RR$. Define $\cH_r=(\cV_r, \langle\fdot,\fdot\rangle_r)$ with
$$\cV_r=\cD ({\bf A}^{r/2})$$
and
$$\langle x,y\rangle_r=\langle {\bf A}^{r/2}x,{\bf A}^{r/2}y\rangle \quad\text{for} \quad x,y\in \cV_r.$$
Then $\cH_r$ is said to be the $r$-th {\em left-definite space} associated with the pair $(\cH,{\bf A})$.
\end{definition}

It was proved in \cite[Theorem 3.1]{LW02} that $\cH_r=(\cV_r, \langle\fdot,\fdot\rangle_r)$ is also described as the left-definite space associated with the pair $(\cH, {\bf A}^r)$.
Specifically, we have:
\begin{enumerate}
\item $\cH_r$ is a Hilbert space,
\item $\cD ({\bf A}^r)$ is a subspace of $\cV_r$,
\item $\cD ({\bf A}^r)$ is dense in $\cH_r$,
\item $\langle x,x\rangle_r\geq k^r\langle x,x\rangle$ for $x\in \cV_r$, and
\item $\langle x,y\rangle_r=\langle {\bf A}^rx,y\rangle$  for $x\in\cD ({\bf A}^r)$, $y\in \cV_r$.
\end{enumerate}

The left-definite domains are defined as the domains of compositions of the self-adjoint operator ${\bf A}$, but the operator acting on this domain is slightly more difficult to define. 

\begin{definition}[{\cite[Definition 2.2/2.3]{LW02}}]
Let $\cH=(\cV,\langle\fdot,\fdot\rangle)$ be a Hilbert space. Suppose ${\bf A}:\cD ({\bf A})\subset \cH\to \cH$ is a self-adjoint operator that is bounded below by $k>0$. Let $1\leq r\in\RR$. If there exists a self-adjoint operator ${\bf A}_r:\cH_r\to \cH_r$ that is a restriction of ${\bf A}$ from the domain $\cD({\bf A})$ to $\cD({\bf A}_r)$,
we call such an operator an $r$-th {\em left-definite operator associated with $(\cH,{\bf A})$}. For $0<r<1$ we obtain an $r$-th {\em left-definite operator associated with $(\cH,{\bf A})$} analogously, but by taking the closure of the domain $\cD({\bf A})$ with respect to the norm induced by the inner product $\langle\fdot, \fdot\rangle_r$.
\end{definition}

The connection between the $r$-th left-definite operator and the $r$-th composition of the self-adjoint operator ${\bf A}$ can be made explicit: 

\begin{corollary}[{\cite[Corollary 3.3]{LW02}}] \label{t-comppower}
Suppose ${\bf A}$ is a self-adjoint operator in the Hilbert space $\cH$ that is bounded below by $k>0$. For each $r>0$, let $\cH_r=(\cV_r, \langle\fdot,\fdot\rangle_r)$ and ${\bf A}_r$ denote, respectively, the $r$-th left-definite space and the $r$-th left-definite operator associated with $(\cH,{\bf A})$. Then
\begin{enumerate}
\item $\cD ({\bf A}^r)=\cV_{2r}$, in particular, $\cD ({\bf A}^{1/2})=\cV_1$ and $\cD ({\bf A})=\cV_2$;
\item $\cD ({\bf A}_r)=\cD ({\bf A}^{(r+2)/2})$, in particular, $\cD ({\bf A}_1)=\cD ({\bf A}^{3/2})$ and $\cD ({\bf A}_2)=\cD ({\bf A}^2)$.
\end{enumerate}
\end{corollary}

The left-definite theory is particularly important for self-adjoint differential operators that are bounded below, as they are generally unbounded. The theory is trivial for bounded operators, as shown in {\cite[Theorem 3.4]{LW02}}.

Our applications of left-definite theory will be focused on differential operators which possess a complete orthogonal set of eigenfunctions in $\cH$. In {\cite[Theorem 3.6]{LW02}} it was proved that the point spectrum of ${\bf A}$ coincides with that of ${\bf A}_r$, and similarly for the continuous spectrum and for the resolvent set. 
It is possible to say more, a complete set of orthogonal eigenfunctions will persist throughout each space in the Hilbert scale.

\begin{theorem}[{\cite[Theorem 3.7]{LW02}}] \label{t-leftdefortho}
If $\{\f_n\}_{n=0}^{\infty}$ is a complete orthogonal set of eigenfunctions of ${\bf A}$ in $\cH$, then for each $r>0$, $\{\f_n\}_{n=0}^{\infty}$ is a complete set of orthogonal eigenfunctions of the $r$-th left-definite operator ${\bf A}_r$ in the $r$-th left-definite space $\cH_r$.
\end{theorem}

Another perspective on the last theorem is that it gives us a valuable indicator for when a space is a left-definite space for a specific operator. Also, we note that left-definite theory can be extended to bounded below operators by applying shifts. Uniqueness is then given up to the chosen shift.

One of the main applications of the theory is to Sturm--Liouville operators. In particular, if the operator has a complete system of orthogonal eigenfunctions then proving boundedness from below is easier and it is possible to invoke Theorem \ref{t-leftdefortho}. An example of how left-definite theory can be applied to such an operator is now briefly described, full details can be found in \cite[Section 12]{LW02}.

\begin{example}\label{e-laguerre}
For $\al>-1$, consider the classical self-adjoint Laguerre differential operator ${\bf A}$ acting on $\cH=L^2\left[(0,\infty),x^{\al}e^{-x}\right]$ via
\begin{align*}
    \ell[f](x)=\dfrac{1}{x^{\al}e^{-x}}\left[-x^{\al+1}e^{-x}f'(x)\right]',
\end{align*}
such that $\dom({\bf A})$ possesses the Laguerre polynomials as a complete set of orthogonal eigenfunctions. The $n$th left-definite Hilbert space associated with the pair $(\cH,{\bf A})$, also possessing this complete set of eigenfunctions, is defined as $\cH_n=(\cV_n,\langle\fdot,\fdot\rangle_n)$, where
\begin{align*}
\cV_n:=\left\{f:(0,\infty)\to\CC ~~\bigg|~~ f\in\text{AC}\ti{loc}^{(n-1)}(0,\infty);~ f^{(n)}\in L^2\left[(0,\infty),t^{\al+n}e^{-t}\right]\right\}
\end{align*}
and
\begin{align*}
\langle p,q\rangle_n:=\sum_{j=0}^n b_j(n,k)\int_0^{\infty} p^{(j)}(t)\overline{q^{(j)}(t)}t^{\al+j}e^{-t}dt ~~\text{ for } (p,q\in\mathcal{P}),
\end{align*}
where $\mathcal{P}$ is the space of all (possibly complex-valued) polynomials. The constants $b_j(n,k)$ are defined as
\begin{align*}
b_j(n,k):=\sum_{i=0}^j \dfrac{(-1)^{i+j}}{j!}\binom{j}{i}(k+i)^n.
\end{align*}
\end{example}

For several years after the discovery of the general left-definite theory, descriptions of left-definite spaces were similar to those of Example \ref{e-laguerre}, i.e.~the boundary conditions were not classically expressed by GKN theory. Similar results for specific (mostly classical) operators can be found in \cite{BLTW, EKLWY, ELT} and their references. Some progress towards expressing left-definite spaces in terms of these standard boundary conditions was made much later in \cite{LW15} and then expanded upon in \cite{FFL}. 

In order to present the main result from \cite{FFL}, we let ${\bf L}^n$ be a self-adjoint operator defined by left-definite theory on $L^2[(a,b),w]$ with domain $\cD_{\bf L}^n$ that includes a complete system of orthogonal eigenfunctions. Enumerate the orthogonal eigenfunctions as $\{P_k\}_{k=0}^{\infty}$. Let ${\bf L}^n$ operate on its domain via $\ell^n[\fdot]$, a differential operator of order $2n$, $n\in\NN$, generated by composing a Sturm--Liouville differential expression with itself $n$ times. Furthermore, let ${\bf L}^n$ be an extension of the minimal operator $L^n\ti{min}$ that has deficiency indices $(n,n)$, and the associated maximal domain be denoted by $\cD\ti{max}^n$. See the Appendix for the general definitions of these domains.

This allows us to compare several different potential descriptions of the left-definite domain. Consider
\begin{align*} 
\cA_n&:=\Big\{f\in \cD\ti{max}^n~:~f,f',\dots,f^{(2n-1)}\in AC\ti{loc}(a,b);
\\
&\hspace{4.8cm}(p(x))^n f^{(2n)}\in L^2[(a,b),w]\Big\},
\\
\cB_n&:=\left\{f\in \cD\ti{max}^n~:~
[f,P_j]_n\Big|_a^b=0 \text{ for }j=0,1,\dots,n-1\right\}, 
\\
\cC_n&:=\left\{f\in \cD\ti{max}^n~:~
[f,P_j]_n\Big|_a^b=0 \text{ for any }n \text{ distinct }j\in\NN \right\}, \text{ and}
\\
\cF_n&:=\left\{f\in \cD\ti{max}^n~:~\left[a_j(x)
f^{(j)}(x)\right]^{(j-1)}\Big|_a^b=0 \text{ for }j=1,2,\dots,n
\right\}.
\end{align*}

The function $p(x)$ above is from the standard definition of a Sturm-Liouville differential operator, given in equation \eqref{d-sturmop}, and the $a_j(x)$'s are from the Lagrangian symmetric form of the operator in equation \eqref{e-lagrangian}. The following conjecture about the equality of these domains is found in \cite{FFL}, which was in turn adapted from \cite{LW15}.

\begin{conjecture}\label{c-equalityofdomains}
Let ${\bf L}^n$ be a self-adjoint operator defined by left-definite theory on $L^2[(a,b),w]$ with domain $\cD_{\bf L}^n$ that includes a complete system of orthogonal polynomial eigenfunctions, that is, we use $\cD_{\bf L}^n = \cA_n$. Let ${\bf L}^n$ operate on its domain via the expression $\ell^n[\fdot]$, a differential operator of order $2n$, where $n\in\NN$, generated by composing a Sturm--Liouville differential operator with itself $n$ times. Furthermore, let ${\bf L}^n$ be an extension of the minimal operator $L^n\ti{min}$, which has deficiency indices $(n,n)$. Then $\cA_n=\cB_n=\cC_n=\cF_n=\cD_{\bf L}^n$, $\forall n\in\NN$.
\end{conjecture}

The conjecture was partially answered in \cite[Theorem 6.5]{FFL} under several extra assumptions: that $\cA_n=\cB_n$ and that $f\in \cF_n$ implies that $f'',\dots,f^{(2n-2)}\in L^2(a,b)$. The two primary ideas of the proof were a careful analysis of the sesquilinear form for the operator ${\bf L}^n$ and the introduction of a matrix of boundary values to help determine when GKN conditions were satisfied or not. The conjecture was answered in the affirmative for the Jacobi differential operator in \cite{FL}.

\begin{theorem}\label{t-leftdefdomains}
Let ${\bf L}^n$, $n\in\NN$, be a self-adjoint operator defined by left-definite theory on $L^2[(a,b),w]$ with the left-definite domain $\cD_{\bf L}^n$. Let ${\bf L}^n$ operate on its domain via $\ell^n[\fdot]$, a classical Jacobi differential expression of order $2n$, with parameters $\al,\beta>0$, generated by composing the Sturm--Liouville operator with itself $n$ times. Furthermore, let ${\bf L}^n$ be an extension of the minimal operator $L^n\ti{min}$, which has deficiency indices $(2n,2n)$. Then $\cD_{\bf L}^n=\cA_n=\cB_n=\cC_n=\cF_n$, $\forall n\in\NN$.
\end{theorem}

Note that the apparent discrepancy between deficiency indices here is merely due to the use of separated boundary conditions instead of connected ones. It is also worth pointing out that the methods of \cite{FL} differ greatly from those of \cite{FFL}. In particular, solutions to the Jacobi differential equation are written as infinite sums and only certain terms are shown to belong to the maximal domain modulo the minimal domain. This finite, explicit decomposition of the deficiency spaces result in easier manipulations of the sesquilinear form to prove the equality of the domains. Conjecture \ref{c-equalityofdomains} remains open for other operators; the methods of \cite{FL} are almost certainly applicable to wider classes of operators and may be helpful. In particular, operators which do not possess a complete system of orthogonal polynomial eigenfunctions have not been considered. A possible weakening of this hypothesis would be to replace the orthogonal polynomials with principal solutions. Principal solutions are intimately related to the Friedrichs extension of a symmetric operator, which is the subject of our last conjecture.

The Friedrichs extension is usually defined through the closed semi-bounded form associated with a self-adjoint operator, see Section \ref{s-BKV} (specifically equation \eqref{e-semiformgen}) for more about these forms or \cite{BdS, BFL, MZ} for complete details. However, it suffices to think of the extension as the ``smallest'' self-adjoint extension among all other self-adjoint extensions (in the sense that it has the smallest form domain); it is often called the ``soft'' extension for this reason.

\begin{conjecture}\label{Friedrichs}
Let $A$ be a closed semi-bounded symmetric operator and ${\bf A}_F$ be its Friedrichs self-adjoint extension. Then the $2r$-th left-definite space of ${\bf A}_F$ coincides with the domain of the Friedrichs extension of the $r$-th power of ${\bf A}$. Explicitly, 
\begin{align*}
\dom(({\bf A}_F)^r)=\dom(({\bf A}^r)_F).
\end{align*}
\end{conjecture}

In other words, the conjecture is suggesting that the action of taking powers of an operator commutes with the action of taking the Friedrichs extension. In every computed case the conjecture seems to hold but in the stated generality the status is unclear. Verification in the case where $A$ is the Jacobi differential operator can be found in \cite[Cor.~5.1]{F}. Left-definite theory need not be mentioned in the statement of the conjecture, clearly, but it is the authors' opinion that left-definite theory can nonetheless be of great help in proving the statement.

A related conjecture (that would build upon the spectral stability results of Theorem \ref{t-leftdefortho} if confirmed) posits that multiplicity of eigenvalues are invariant under left-definite theory.

\begin{conjecture}\label{c-multiplicity}
Let ${\bf A}$ be a semi-bounded self-adjoint operator in a Hilbert space $\cH$. Let $\la$ be an eigenvalue of ${\bf A}$ with multiplicity $m$. Then $\la$ is also an eigenvalue for the $r$-th left-definite operator ${\bf A}_r$ in the $r$-th left-definite space $\cH_r$ of multiplicity $m$.
\end{conjecture}

The difficulty of determining multiplicity of eigenvalues restricts the amount of evidence available to support the conjecture. In the case when ${\bf A}$ is the Jacobi differential operator, left-definite operators and Weyl $m$-functions for their extensions can be found in \cite{F} and the spectral analysis of \cite{BFL} could be applied to potentially verify the conjecture for the example. 

Conjecture \ref{c-multiplicity} also has implications for the intertwining of eigenvalues between distinct extensions. Usually this intertwining takes place between the Friedrichs extension and a transversal extension, which can be seen e.g.~in a comparison of Neumann and Dirichlet eigenvalues for Sturm--Liouville operators with two regular endpoints. Essentially, it would be interesting to see if eigenvalue intertwining between two self-adjoint extensions implied that the $r$-th left-definite operators of the two extensions also had eigenvalue intertwining.

\section{Comparison with BKV Semi-Bounded Form Theory}\label{s-BKV}

We first introduce the basics of the so-called Birman--Krein--Vishik (BKV) theory of semi-bounded forms. A collection of results from the theory along with original references can be found in e.g.~\cite{AS}. The brief presentation here mostly follows that of \cite[Chapter 5]{BdS} and \cite[Chapter 6]{Kato}, which can also be consulted for more details.

Let ${\bf A}$ be a semi-bounded self-adjoint operator with lower bound $m({\bf A})<\infty$. There is a natural way to identify ${\bf A}$ with a closed semi-bounded form $\ft$ in $\cH$ with the same lower bound $m(\ft)=m({\bf A})$ via the First and Second Representation Theorems, see e.g.~\cite[Theorem 5.1.18]{BdS} and \cite[Theorem 5.1.23]{BdS}. Namely, let $\f\in\dom ({\bf A})$, $\psi\in\dom(\ft)$, $\gamma<m({\bf A})$ and define
\begin{align*}
    \dom(\ft_{{\bf A}})&=\dom({\bf A}-\gamma)^{1/2}, \\
    \ft_{{\bf A}}[\f,\psi]&=\langle ({\bf A}-\gamma)\f,\psi\rangle+\gamma\langle\f,\psi\rangle.
\end{align*}
Equivalently, if $\f,\psi\in\dom(\ft_{{\bf A}})$, then
\begin{align*}
    \ft_{{\bf A}}[\f,\psi]=\langle ({\bf A}-\gamma)^{1/2}\f,({\bf A}-\gamma)^{1/2}\psi\rangle +\gamma\langle\f,\psi\rangle.
\end{align*}
The space $\dom(\ft_{{\bf A}})$ endowed with the inner product
\begin{align}\label{e-forminnerprod}
    \langle\f,\psi\rangle\ci{\ft_{{\bf A}-\gamma}}:=\ft_{{\bf A}}[\f,\psi]-\gamma\langle\f,\psi\rangle, \text{ for }\f,\psi\in\dom(\ft),
\end{align}
is a Hilbert space, denoted $\cH\ci{\ft_{{\bf A}-\gamma}}$.

More general semi-bounded symmetric operators $S$ (i.e.~the minimal operator of a Sturm--Liouville expression) also have a form associated with them via 
\begin{align}\label{e-semiformgen}
    \ft_S[f,g]=\langle Sf,g\rangle, \text{ for }f,g\in\dom(S).
\end{align}
The semi-bounded self-adjoint operator ${\bf S}_F$ associated with the closure of the form $\ft_S$ in equation \eqref{e-semiformgen} is called the {\em Friedrichs extension} of $S$, see \cite[Definition 5.3.2]{BdS}.

This theory is often used to distinguish or construct specific self-adjoint extensions from a symmetric operator, as it is often more convenient to define a closed semi-bounded form than a self-adjoint operator. An ordering of closed semi-bounded forms then corresponds to an ordering of self-adjoint extensions, is also useful for this purpose. Details can be found in e.g.~\cite[Section 5.2]{BdS} and \cite[Remark 3.5]{BFL}.

It is also possible to put densely defined, closed, sectorial forms $\ft_T$ into one-to-one correspondence with $m$-sectorial operators $T$, but this falls outside the scope of the current manuscript (see e.g.~\cite{Arl, BT, C}). However, this shows that the correspondence between forms and operators is usable in a wide variety of contexts. 

In comparison with Section \ref{s-LD}, it should be clear that given a semi-bounded self-adjoint operator ${\bf A}$ the first left-definite space automatically coincides with the domain of the associated closed semi-bounded form after an appropriate shift to make the operator positive, i.e.~$\cV_1=\dom(\ft_{\bf A})$. Indeed, for $f,g\in\cV_1$ the action of $\ft_{\bf A}[f,g]$ is equal to $\langle f,g\rangle_1$, the inner product in the first left-definite space. 

If we restrict our attention to closed semi-bounded forms, left-definite theory is a natural extension of BKV theory; instead of associating a single closed semi-bounded form with a self-adjoint operator, it is possible to associate a whole continuum of closed semi-bounded forms. 

\begin{proposition}\label{p-new}
Let ${\bf A}$ be a semi-bounded self-adjoint operator with lower bound $m({\bf A})$. For each $r\in\NN$, the form (we suppress the dependence on ${\bf S}$ here), with $f,g\in\dom(\ft_r)$ and $\gamma<m({\bf A})$, given by
\begin{align*}
    \dom(\ft_{r})&=\dom({\bf A}-\gamma)^{r/2}, \\
    \ft_r[f,g]&=\langle(({\bf A}-\gamma)^{r/2}f,({\bf A}-\gamma)^{r/2}g\rangle+\gamma\langle f,g\rangle,
\end{align*}
is closed and semi-bounded. 
\end{proposition}

\begin{proof}
Here, we consider the left-definite spaces associated with the shifted operator ${\bf A}-\gamma$. Semi-boundedness of $\ft_r$ is assured by item (4) after Definition \ref{t-ldinpro}. Lemma 5.1.9 of \cite{BdS} says that the semi-bounded form $\ft_r$ is closed if and only if the space $\cH\ci{({\bf A}-\gamma)^r}$ is in fact a Hilbert space. Endow the space $\dom(\ft_r)$ with the inner product $\langle\cdot,\cdot\rangle_r$ from Definition \ref{t-ldinpro} and notice that this coincides with the definition of the space $\cH\ci{({\bf A}-\gamma)^r}$ above. Since $\cH_r$ is a Hilbert space, this immediately implies that $\ft_r$ is closed. 
\end{proof}

While the BKV theory of semi-bounded forms gives valuable knowledge about self-adjoint extensions, in practice the domain of the operator ${\bf A}^{1/2}$ is often difficult to determine, even for elementary choices of ${\bf A}$. Explicit domains can be determined when ${\bf S}$ is a Sturm--Liouville operator, see \cite[Section 6.9]{BdS}. Left-definite domains have some advantages and disadvantages in this area. 

As far as advantages go, for even $r$, left-definite domains do not involve fractional powers of operators and therefore are somewhat natural to consider. The square of a self-adjoint operator is sometimes useful in applications and falls into this category. There is also some spectral stability inherent in left-definite operators, see Theorem \ref{t-leftdefortho}, that is, to the best knowledge of the authors, not available in BKV form theory. Left-definite theory thus provides alternatives to the classical form used in BKV theory and these form domains can be expressed with classical boundary conditions, see e.g.~Theorem \ref{t-leftdefdomains}.

The main disadvantage of the theory is that it is unclear whether the ordering of closed semi-bounded forms is preserved under left-definite theory. This is made more difficult by the fact that explicit boundary conditions are somewhat elusive, see Theorem \ref{t-leftdefdomains} for an example where they were determined. A proof of Conjecture \ref{Friedrichs} would go a long way towards solving this discrepancy. BKV theory, as previously mentioned, is also applicable to a wider range of operators.

It is also important to note that, in contrast to Section \ref{s-LD}, the discussion here was not concerned with differential operators but holds in the wider context of semi-bounded self-adjoint operators.

BKV theory was exploited in \cite{BFL} to build a boundary pair for Sturm--Liouville operators with limit-circle endpoints that was compatible with a boundary triple. This allows the set up of a perturbation problem that describes all possible self-adjoint extensions of the minimal operator by using a scale of spaces in the following section.

\section{Scale of Spaces from Singular Perturbation Theory}\label{s-SCALE}

Let ${\bf A}$ and ${\bf T}$ be operators on a Hilbert space $\cH$. In perturbation theory, we are interested in the following question: If we know the properties of operator ${\bf A}$ well, what can we say about the formal operator ${\bf A}+{\bf T}$?

When the Hilbert space $\cH$ is infinite, an immediate question is the rigorous definition what the sum of two operators ${\bf A}+{\bf T}$. To briefly illustrate the severity of potential problems, recall that, by the closed graph theorem, we know that the unbounded operators ${\bf A}$ and ${\bf T}$ are only defined on a dense subset of $\cH$. Noticing this, we realize that it can easily happen that the intersection of their domains is empty. In this case, ${\bf A}+{\bf T}$ would not be interesting as its domain equals the empty set. Of course, less severe scenarios can also lead to serious issues with the meaning of ${\bf A}+{\bf T}$.

Throughout perturbation theory, this problem is dealt with by making situation or application dependent assumptions on ${\bf A}$ and/or ${\bf T}$. Most frequently, some smallness hypothesis on the perturbation ${\bf T}$ is imposed. In our application to Sturm--Liouville operators, both ${\bf A}$ and ${\bf T}$ will be unbounded self-adjoint operators (making their sum self-adjoint as well), though ${\bf A}$ will be bounded from below and ${\bf T}$ will be of finite rank, i.e.~has finite dimensional range. The fact that ${\bf T}$ is of finite rank, will allows us to formulate this perturbation problem rather concretely. With this setup, the operator ${\bf A}$ will be one self-adjoint extensions of a Sturm--Liouville operator and the perturbed operators ${\bf A}+{\bf T}$ will stand in bijection to all possible self-adjoint extensions of the minimal operator with ${\bf T}$ encoding those boundary conditions.

Further, in our application the following key property holds: the range of ${\bf T}$ is contained in a Hilbert space generated by operator ${\bf A}$.
It will turn out that ${\bf T}$ is relatively bounded with respect to ${\bf A}$, so that the domain of ${\bf A}+{\bf T}$ equals that of ${\bf A}$.
To carry out this plan, we now define this scale of Hilbert spaces and then discuss the meaning of ${\bf A}+{\bf T}$ in our situation.

The following definition of these finite rank singular form bounded perturbations roughly follows that in \cite{AK}. Let ${\bf A}$ be a self-adjoint operator on $\cH$.
Consider the non-negative operator $|{\bf A}|=({\bf A}^*{\bf A})^{1/2}$, whose domain coincides with the domain of ${\bf A}$. We introduce a scale of Hilbert spaces.

\begin{definition}[\hspace{-1pt}{see, e.g.~\cite[Section 1.2.2]{AK}}]\label{d-standardscale}
For $s\geq 0$, define the space $\cH_s(A)$ to consist of $\f\in\cH$ for which the $s$-norm
\begin{align}\label{d-scalenorm}
\|\f\|_s:=\|(|{\bf A}|+I)^{s/2}\f\|_{\cH},
\end{align}
is bounded. 
The space $\cH_s({\bf A})$ equipped with the norm $\|\cdot\|_s$ is complete. The adjoint spaces, formed by taking the linear bounded functionals on $\cH_s({\bf A})$, are used to define these spaces for negative indices, i.e.~$\cH_{-s}({\bf A}):=\cH_s^*({\bf A})$. The corresponding norm in the space $\cH_ {-s}({\bf A})$ is thus defined by \eqref{d-scalenorm} as well. 
The collection of these $\cH_s({\bf A})$ spaces will be called the \emph{scale of Hilbert spaces associated with the self-adjoint operator ${\bf A}$}.
\end{definition}

Alternatively, if ${\bf A}$ is semi-bounded with lower bound $m({\bf A})$, then we can choose to consider ${\bf A}-\gamma$ for $\gamma<m({\bf A})$ instead of $|{\bf A}|+I$. In particular, both options generate the same spaces with equivalent norms.

It is not difficult to see that the spaces satisfy the nesting properties
\begin{align*}
\hdots\subset\cH_2({\bf A})\subset\cH_1({\bf A})\subset\cH=\cH_0({\bf A})\subset\cH_{-1}({\bf A})\subset\cH_{-2}({\bf A})\subset\hdots,
\end{align*}
and that for every two $s,t$ with $s<t$, the space $\cH_t({\bf A})$ is dense in $\cH_s({\bf A})$ with respect to the norm $\|\cdot\|_s$. Indeed, the operator $({\bf A}+I)^{t/2}$ defines an isometry from $\cH_s({\bf A})$ to $\cH_{s-t}({\bf A})$. For $\f\in\cH_{-s}({\bf A})$, $\psi\in\cH_s({\bf A})$, we define the duality pairing
\begin{align*}
\langle\f,\psi\rangle_{s,-s}:=\big\langle(|{\bf A}|+I)^{-s/2}\f,(|{\bf A}|+I)^{s/2}\psi\big\rangle.
\end{align*}

The main perspective we are choosing here is the relations of these Hilbert scales with those spaces generated by left-definite theory, see Section \ref{s-LD}. Outside of this perspective we mention on the side that throughout the literature of other fields similar constructions occur under different names. For instance, the pairing of $\cH_1({\bf A})$, $\cH$, and $\cH_{-1}({\bf A})$ is sometimes referred to as a \emph{Gelfand triple} or \emph{rigged Hilbert space}. Also, when $\bA$ is the derivative operator, these scales are simply Hilbert--Sobolev spaces. When ${\bf A}$ is a general  differential operator, they are closely related. More details about Hilbert scales can be found in \cite{KP}.

Aside, we also mention that finite-rank perturbations of a given operator $\bA$ arise most commonly when the vectors $\f$ are bounded linear functionals on the domain of the operator ${\bf A}$; so, many applications are focused on $\cH_{-2}({\bf A})$. Here, we  discuss the case $\f\in\cH_{-1}({\bf A})$ for the sake of simplicity, the so-called form bounded singular case. However, \cite{AK} contains information on extensions to $\f\in\cH_{-2}({\bf A})$, and the case when $\f\notin\cH_{-2}({\bf A})$ can be found in \cite{DKS, Kurasov}.

To define rank-$d$ form bounded perturbations of a self-adjoint operator $\bA$ on a Hilbert space $\cH$, consider a coordinate mapping ${\bf B}:\CC^d\to\Ran({\bf B})\subset\cH_{-1}\left({\bf A}\right)$ that acts via multiplication by the row vector
\begin{align*}
    \begin{pmatrix}
    f_1, ~~\hdots, f_d\\
    \end{pmatrix}
    \qquad\text{with } f_1, \hdots, f_d\in \cH_{-1}\left({\bf A}\right).
\end{align*}
Formally, the mapping ${\bf B}^*:\Ran({\bf B})\to\CC^d$ acts by
\begin{align*}
    {\bf B}^*\fdot=
    \begin{pmatrix}
    \langle \fdot,f_1\rangle_{s,-s}\\
    \vdots\\
    \langle \fdot,f_d\rangle_{s,-s}
    \end{pmatrix}.
\end{align*}
We say \emph{formally}, because the inner products occurring in ${\bf B}^*$ are not defined on all of $\Ran({\bf B})$. However, in accordance with the definition of $\cH_{-1}\left({\bf A}\right)$, they do make sense as a duality pairing on the quadratic form space of the unperturbed operator $\bA$. And that is all we need. Abusing notation slightly, we use the same notation ${\bf B}^*$ for the operator restricted to this form domain.

The quadratic form sense now gives rigorous meaning to the finite rank form bounded singular perturbation
\begin{align}\label{e-definition}
    \bA\ci\Theta:=\bA+{\bf B}\Theta{\bf B}^*,
\end{align}
where $\Theta:\C^d\to\C^d$ is an $d\times d$ matrix (not a linear relation).

For the interpretation and application it is easiest to fix a coordinate map ${\bf B}$. The definition in equation \eqref{e-definition} can be extended to linear relations  $\Theta$. It is well-known that boundary conditions then stand in bijection to the self-adjoint linear relation $\Theta$. One way to access the explicit translation between boundary conditions and $\Theta$ is via boundary triplets. We decided not to include this information due to accessibility, see e.g.~\cite{BFL}. More information on finite-rank perturbations can be found in e.g.~\cite{FL2, LT_JST, LTSurvey}.

Given this setup, the main problem usually becomes determining which space a desired perturbation vector comes from. In practice, it can be very difficult to actually compute the norm in Definition \ref{d-standardscale} so it does not appear there are many tools available for this purpose. However, connections with left-definite theory and BKV semi-bounded form theory can be exploited here.

For a semi-bounded self-adjoint operator ${\bf A}$ with lower bound $m({\bf A})$, choose $\gamma<m({\bf A})$. Given the note after Definition \ref{d-standardscale}, it is then clear that $\cH_{1}({\bf A}-\gamma)$ coincides with $\cH\ci{{\bf A}-\gamma}$ (inner product given in equation \eqref{e-forminnerprod}) so that the underlying spaces are the same. Hence, showing that a perturbation vector belongs to the class $\cH_{-1}({\bf A}-\gamma)$ is the same as showing it is a linear bounded functional on the first left-definite space $\cV_1$ or the form domain $\ft_{{\bf A}-\gamma}$.

Here we find another limitation of BKV semi-bounded form theory, as it is somewhat natural to consider what spaces are analogously related for more singular perturbations. Fortunately, the connection with left-definite theory remains valid in these cases and provides good indicators for what spaces perturbation vectors might be in. For instance, a perturbation vector belongs to the class $\cH_{-2}({\bf A}-\gamma)$ if and only if it is a linear bounded functional on $\cH_{2}({\bf A}-\gamma)$, whose underlying space is simply the second left-space $\cV_2$, and so on and so forth for higher values of $s$ in accordance with Corollary \ref{t-comppower}. Left-definite theory is thus a fundamental subject in the study of singular perturbation theory. 

Singular perturbation theory is thus complementary to left-definite theory. It provides a rigorous analysis of linear bounded functionals on left-definite spaces, thereby expanding the framework to consider $-r\in\NN$, and many useful applications. However, because singular perturbation theory is for general self-adjoint operators, it lacks the key spectral stability results of Theorem \ref{t-leftdefortho} and explicit descriptions of the domains involved. If the self-adjoint operator is strictly positive, both theories can be applied directly.

\section{Perturbation setup}\label{s-PERT}

The connection between BKV semi-bounded form theory and singular perturbation theory was recently used by the authors in \cite{BFL} to obtain a characterization of all possible self-adjoint extensions of Sturm--Liouville differential operators with one or two limit-circle endpoints. As explained in the previous two subsections, the connection these theories have with left-definite theory mean that the left-definite space $\cV_1$ is also inherently involved. 

The perturbation setup itself is a bit technical, so we refer the reader to the full manuscript \cite{BFL} for details and cover only the broad strokes here. 

The operators ${\bf B}$ and ${\bf B}^*$ are as in Section \ref{s-PERT} with the perturbation vectors $f_1$ and $f_2$ chosen to have their duality pairing act like the sesquilinear form for the Sturm--Liouville operator with one input being the principal solution. Namely, $f_1$ and $f_2$ are elements from $\cH_{-1}({\bf A})$ defined so that
\begin{align}\label{e-deltaexplain}
    \langle\fdot,f_1\rangle_{1,-1}:=[\fdot,u_a](x)\Big|_{x=a} \quad\text{and}\quad
    \langle\fdot,f_2\rangle_{1,-1}:=[\fdot,u_b](x)\Big|_{x=b},
\end{align}
where $u_a$ and $u_b$ are principal solutions to the eigenvalue problem for some $\la\in\RR$ near the endpoints $x=a$ and $x=b$ of the differential equation, respectively. The self-adjoint extension ${\bf A}_0$ of the minimal operator is chosen to be transversal to the domain of the Friedrichs extension and defined via a boundary triple. Namely, if ${\bf A}_F$ is the Friedrichs extension, then the span of $\dom({\bf A}_F)\cup\dom({\bf A}_0)$ is the maximal domain and $\dom({\bf A}_F)\cap\dom({\bf A}_0)$ is the minimal domain.

\begin{theorem}[{\cite[Theorem 3.6]{BFL}}]\label{t-twopert}
Let $\Theta$ be a self-adjoint linear relation in $\CC^2$. Define ${\bf A}\ci\Theta$ as the singular rank-two perturbation:
\begin{align}\label{e-twopert}
    {\bf A}\ci\Theta:={\bf A}_0+{\bf B}\Theta{\bf B}^*.
\end{align}
Then every self-adjoint extension of the minimal operator $L\ti{min}$ can be written as ${\bf A}\ci\Theta$ for some $\Theta$.
\end{theorem}

In an application, the operator ${\bf A}\ci{\Theta}$ being well-defined reduces to showing that the somewhat abstractly defined $f_1,f_2$ belong to $\cH_{-1}({\bf A}_0)$. It turns out that this is essentially a consequence of constructing a boundary pair. In this context, a boundary pair consists of
\begin{itemize}
    \item a bounded operator from $\cH_{1}({\bf A}_0)$ to $\CC^2$ that is surjective and whose kernel is equal to the domain of the Friedrichs extension,
    \item and the space $\CC^2$. 
\end{itemize}
Such a map from $\cH_{1}({\bf A}_0)$ to $\CC^2$ is then a linear bounded functional on $\cH_{1}({\bf A}_0)$ and hence generated by a function from $\cH_{-1}({\bf A}_0)$ via the duality pairing. Indeed, the proof uses this fact, as $f_1$ and $f_2$ are chosen to naturally fit this requirement. 
The other part of the theorem -- that equation \eqref{e-twopert} establishes a one-to-one relationship between self-adjoint extensions and self-adjoint linear relations $\Theta$ -- relies on a parameterization stemming from the theory of boundary triples. Note that the use of linear relations here is in contrast to the matrices used in Section \ref{s-PERT}.

We also note that the perturbation from Theorem \ref{t-twopert} can only be rigorously interpreted by appealing to BKV semi-bounded form theory, see \cite[Remark 3.7]{BFL} for details.

Ideally, the perturbation setup could be adapted to work for powers of Sturm--Liouville operators. This might provide a nice class of examples of finite rank perturbations. The following conjecture states what a formulation would look like.

\begin{conjecture}\label{c-powers}
Let ${\bf L}^n_0$, $n\in\NN$, be a self-adjoint extension of the minimal operator $L^n\ti{min}$ on $L^2[(a,b),w]$, which has deficiency indices $(m,m)$ with $m\in\{n,2n\}$. Let ${\bf L}^n$ operate on its domain via $\ell^n[\fdot]$ (defined in equation \eqref{e-lagrangian}), a semi-bounded Sturm--Liouville differential expression of order $2n$ generated by composing the Sturm--Liouville operator with itself $n$ times. Let $\Theta$ be a self-adjoint relation in $\CC^m$. Then there exist a choice of ${\bf L}^n_0$, ${\bf B}$ and ${\bf B}^*$ such that the singular rank-$m$ perturbation
\begin{align*}
    {\bf L}\ci\Theta:={\bf L}^n_0+{\bf B}\Theta{\bf B}^*.
\end{align*}
is well-defined and every self-adjoint extension of the minimal operator $L\ti{min}$ can be written as ${\bf L}\ci\Theta$ for some $\Theta$.
\end{conjecture}

The parametrization of self-adjoint linear relations again relates to the boundary conditions imposed by an underlying boundary triple here, so obtaining all self-adjoint extensions via the perturbation should not be an overly difficult problem. The choice of the operator ${\bf L}^n_0$ allows for some freedom; it was the result of a boundary triple construction in Theorem \ref{t-twopert} but these are not unique and the only requirement needed here is that the chosen extension is transversal to the Friedrichs domain.

The real issue is finding perturbation vectors $f_1,\dots,f_m$ from $\cH_{-1}({\bf L}_0^n)$ so that ${\bf B}$ and ${\bf B}^*$ are well-defined. In Theorem \ref{t-twopert} these roles were essentially played by principal solutions via equation \eqref{e-deltaexplain}, but an obvious analog does not immediately present itself in this case. For any $\lambda\in\RR$ there exist general solutions to the eigenvalue problem for the uncomposed Sturm--Liouville differential operator near an endpoint, one principal and one non-principal if it is limit-circle. Hence, there are $n$ principal solutions to use in equation \eqref{e-deltaexplain}, but it does not appear that these functions can create a boundary triple like in the base case, see e.g.~\cite[Chapter 6]{BdS}. This problem was solved in \cite{F} for powers of the Jacobi differential operator by applying a modified Gram--Schmidt process to these principal solutions, but a generalization is outstanding. 

A proof of Conjecture \ref{Friedrichs} would also be helpful here, as an analogous statement may be true for the transversal self-adjoint extension ${\bf L}_0$. A possible intermediary goal would then be to state the problem in terms of the uncomposed operator ${\bf L}_0$: the operator ${\bf L}^n_0$ would be the $2n$-th left-definite operator associated with ${\bf L}_0$. Hence, the main problem in Conjecture \ref{c-powers} would be to show that the perturbation vectors in ${\bf B}^*$ are in $\cH_{-n}({\bf L}_0)$. This seemingly small change could have a large impact.

\section*{Appendix: Extension Theory}\label{ss-extensions}

For readers not familiar with classical self-adjoint extension theory for symmetric operators we include some basic definitions and notions as applied to ordinary differential operators. The classical references \cite{AG,N} can also be consulted for further details.

Let $\ell$ be a Sturm--Liouville differential expression. It is important to reiterate that the analysis of self-adjoint extensions does not  at all involve changing the differential expression associated with the operator, merely the domain of definition by applying boundary conditions.

\begin{definition}[{see, e.g.~\cite[Section 17.4]{N}}]\label{d-max}
The {\em maximal domain} of $\ell[\fdot]$ is given by 
\begin{align*}
\cD\ti{max}=\cD\ti{max}(\ell):=\big\{f:(a,b)\to\mathbb{C}~:~f,pf'&\in\text{AC}\ti{loc}(a,b); \\
&f,\ell[f]\in L^2[(a,b),w]\big\}.
\end{align*}
\end{definition}

The designation of ``maximal'' is appropriate in this case because $\cD\ti{max}(\ell)$ is the largest possible subset of $L^2[(a,b),w]$ that $\ell$ maps back into $L^2[(a,b),w]$. For $f,g\in\cD\ti{max}(\ell)$ and $a<a_0\le b_0<b$ the {\em sesquilinear form} associated with $\ell$ by 
\begin{equation}\label{e-greens2}
[f,g]\bigg|_{a_0}^{b_0}:=\int_{a_0}^{b_0}\left\{\ell[f(x)]\overline{g(x)}-\ell[\overline{g(x)}]f(x)\right\}w(x)dx.
\end{equation}
The notation for the sesquilinear form does not involve $\ell$ explicitly, rather the differential expression will be clear from context.

\begin{theorem}[{see, e.g.~\cite[Section 17.4]{N}}]\label{t-limits}
The limits $[f,g](b):=\lim_{x\to b^-}[f,g](x)$ and $[f,g](a):=\lim_{x\to a^+}[f,g](x)$ exist and are finite for $f,g\in\cD\ti{max}(\ell)$.
\end{theorem}

The equation \eqref{e-greens2} is {\em Green's formula} for $\ell[\fdot]$, and in the case of Sturm--Liouville operators \eqref{d-sturmop} it can be explicitly computed using integration by parts to be the modified Wronskian
\begin{align*}
[f,g]\bigg|_a^b:=p(x)[f'(x)g(x)-f(x)g'(x)]\bigg|_a^b.
\end{align*}

\begin{definition}[{see, e.g.~\cite[Section 17.4]{N}}]\label{d-min}
The {\em minimal domain} of $\ell[\fdot]$ is given by
\begin{align*}
\cD\ti{min}=\cD\ti{min}(\ell):=\left\{f\in\cD\ti{max}(\ell)~:~[f,g]\big|_a^b=0~~\forall g\in\cD\ti{max}(\ell)\right\}.
\end{align*}
\end{definition}

The maximal and minimal operators associated with the expression $\ell[\fdot]$ are then defined as $L\ti{min}=\{\ell,\cD\ti{min}\}$ and $L\ti{max}=\{\ell,\cD\ti{max}\}$, respectively. By {\cite[Section 17.2]{N}}, these operators are adjoints of one another, i.e.~$(L\ti{min})^*=L\ti{max}$ and $(L\ti{max})^*=L\ti{min}$. The operator $L\ti{min}$ is thus symmetric.

Note that the self-adjoint extensions of a symmetric operator coincide with those of the closure of the symmetric operator {\cite[Theorem XII.4.8]{DS}}, so without loss of generality we assume that $L\ti{min}$ is closed.

\begin{definition}[variation of {\cite[Section 14.2]{N}}]\label{d-defect}
Define the {\em positive defect space} and the {\em negative defect space}, respectively, by
$$\cD_+:=\left\{f\in\cD\ti{max}~:~L\ti{max}f=if\right\}
\quad\text{and}\quad
\cD_-:=\left\{f\in\cD\ti{max}~:~L\ti{max}f=-if\right\}.$$
\end{definition}

The dimensions dim$(\cD_+)=m_+$ and dim$(\cD_-)=m_-$, called the {\em positive} and {\em negative deficiency indices of $L\ti{min}$} respectively, will play an important role. They are usually conveyed as the pair $(m_+,m_-)$. The symmetric operator $L\ti{min}$ has self-adjoint extensions if and only if its deficiency indices are equal {\cite[Section 14.8.8]{N}}.

The classical {\em von Neumann formula} then says that 
\begin{align*}
\cD\ti{max}=\cD\ti{min}\dotplus\cD_+\dotplus\cD_-.
\end{align*}
The decomposition can be made into an orthogonal direct sum by using the graph norm, see \cite{FFL}.
If the operator $L\ti{min}$ has any self-adjoint extensions, then the deficiency indices of $L\ti{min}$ have the form $(m,m)$, where $0\leq m\leq 2$ {\cite[Section 14.8.8]{N}}.
Hence, Sturm--Liouville expressions that generate self-adjoint operators must have deficiency indices $(0,0)$, $(1,1)$ or $(2,2)$. If a differential expression is either in the limit-circle case or regular at the endpoint $a$, it requires a boundary condition at $a$. If it is in the limit-point case at the endpoint $a$, it does not require a boundary condition. The analogous statements are true at the endpoint $b$.


\end{document}